\documentclass[12pt,a4paper]{article}

\usepackage{authblk}
\usepackage[margin=2cm]{geometry}
\usepackage{t1enc}
\usepackage[utf8]{inputenc}
\usepackage{amsthm,amsmath,amssymb}
\usepackage{graphicx}
\usepackage{enumerate}
\usepackage{hyperref}
\usepackage{bm}
\usepackage{comment}
\usepackage{amsfonts}
\usepackage{graphicx,caption}
\usepackage{bm}
\usepackage{amsmath, amsthm, amssymb}
\usepackage{graphicx}
\usepackage{hyperref}
\usepackage{relsize}
\usepackage{blkarray}
\usepackage{algpseudocode}

\usepackage{bbm}

\theoremstyle{plain}
\usepackage{amsthm}
\makeatletter
\newcommand{\newreptheorem}[2]{\newtheorem*{rep@#1}{\rep@title}\newenvironment{rep#1}[1]{\def\rep@title{#2 \ref*{##1}}\begin{rep@#1}}{\end{rep@#1}}}
\makeatother

\newtheorem{theorem}{Theorem}
\newtheorem*{theorem-non}{Theorem}
\newtheorem*{non-lemma}{Lemma}
\newtheorem{lemma}[theorem]{Lemma}
\newreptheorem{lemma}{Lemma}

\newtheorem{corr}[theorem]{Corollary}

\newtheorem{conjecture}[theorem]{Conjecture}
\theoremstyle{definition}
\newtheorem{remark}[theorem]{Remark}

\DeclareMathOperator{\sym}{sym}

\DeclareMathOperator{\torsion}{Torsion\,}

\DeclareMathOperator{\mg}{mg}

\DeclareMathOperator{\Aut}{Aut}

\DeclareMathOperator{\gr}{graph\,}

\DeclareMathOperator{\dkl}{D_{KL}}

\begin{document}
\title{Using dense graph limit theory to count cocycles of random simplicial complexes}
\author{Andr\'as M\'esz\'aros}
\date{}
\affil{HUN-REN Alfr\'ed R\'enyi Institute of Mathematics}
\maketitle
\begin{abstract} 
We develop a limit theory for $1$-cochains of complete graphs with coefficients from a finite abelian group. We prove an analogue of the large deviation principle of Chatterjee and Varadhan for random cochains. We use these new tools to prove results about the homology of random $2$-dimensional simplicial complexes. 

More specifically, we prove that if $T_n$ is a random $2$-dimensional determinantal hypertree on $n$ vertices and $p$ is any prime, then
\[\frac{\dim H_1(T_n,\mathbb{F}_p)}{n^2}\]
converges to zero in probability. The same result holds for random $1$-out 2-complexes.
\end{abstract}

\section{Introduction}

Let $X_n$ be a sequence of random $2$-dimensional simplicial complexes on $n$ vertices such that the $1$-skeleton of $X_n$ is the complete graph $K_n$. Assume that for an abelian group $G$, we would like to give an upper bound on
\begin{equation}\label{EZint}\frac{\log \mathbb{E} |Z^1(X_n,G)|}{n^2}=\frac{\log \sum_{f\in C^1(K_n,G)} \mathbb{P}(f\in Z^1(X_n,G)) }{n^2}.\end{equation}

If for all $u\ge 0$, we can get a good upper bound on the number of cochains $f\in C^1(K_n,G)$ such that $\log \mathbb{P}(f\in Z^1(X_n,G))\ge -un^2$, then we can combine these bounds to get an estimate on \eqref{EZint}. For $G=\mathbb{Z}/2\mathbb{Z}$, any cochain in $C^1(K_n,G)$ is uniquely determined by its support. Therefore, in this case, $C^1(K_n,G)$ can be identified with the set of graphs on $n$ vertices. Thus, we need to count the graphs on $n$ vertices such that $\log \mathbb{P}(f\in Z^1(X_n,G))\ge -un^2$ holds for the corresponding cochain $f$. Provided that $\mathbb{P}(f\in Z^1(X_n,G))$ can be expressed in reasonable graph theoretic terms, it can be expected that we can use graph theoretic tools to estimate the number of such graphs. In fact, what we need here is a \emph{large deviation principle} for random graphs. Such a result was proved by Chatterjee and Varadhan \cite{chatterjee2011large}. Their proof relies on \emph{dense graph limit} theory, which was developed by Lov\'asz, Szegedy, T.~S\'os, Borgs, Chayes, Vesztergombi, Schrijver, Freedman and many others \cite{dl1,dl2,dl3,dl4,dl5,dl6,dl7,dl8,dl9,dl10}. See the book of Lov\'asz~\cite{lovasz2012large} for an overview of the area of graph limit theory. 

In the paper \cite{meszaros2024bounds}, the author successfully applied the above ideas to prove that
\[\frac{\dim H_1(X_n,\mathbb{F}_2)}{n^2}\]
converge to $0$ in probability in two specific cases: when $X_n$ is a $2$-dimensional determinantal hypertree, and when $X_n$ is a $1$-out $2$-complex.

For $|G|>2$, it is slightly more complicated to describe $C^1(K_n,G)$ in graph theoretic terms. Namely, a cochain in $C^1(K_n,G)$ corresponds to a complete directed graph on $n$ vertices, where the edges are labeled by elements of $G$ in such a way that if the directed edge $uv$ is labeled by~$g$, then the edge $vu$ is labeled by $-g$. Thus, to generalize the above ideas for $|G|>2$, we need to build a limit theory of such labeled complete directed graphs and prove the analogue of the large deviation principle of Chatterjee and Varadhan. The main goal of this paper is to develop these tools. Our results will allow us to generalize the mod $2$ results of \cite{meszaros2024bounds} to odd characteristics as we describe next. 

\subsection{The mod $p$ homology of random complexes}

\emph{Determinantal hypertrees} are natural higher dimensional generalizations of a uniform random spanning tree of a complete graph. They can be defined in any dimension, but in this paper, we restrict our attention to the $2$-dimensional case. A $2$-dimensional simplicial complex $S$ on the vertex set $[n]=\{1,2,\dots,n\}$ is called a ($2$-dimensional) hypertree, if
\begin{enumerate}[\hspace{30pt}(a)]
 \item\label{pra} $S$ has complete $1$-skeleton;
 \item\label{prb} The number of triangular faces of $S$ is ${n-1}\choose{2}$;
 \item\label{prc} The homology group $H_{1}(S,\mathbb{Z})$ is finite.
\end{enumerate}

In one dimension, a spanning tree must be connected, property \eqref{prc} above is the two dimensional analogue of this requirement. Note that any complex $S$ satisfying \eqref{pra} and \eqref{prc} must have at least ${n-1}\choose{2}$ triangular faces. For a graph $G$, if the reduced homology group $\tilde{H}_0(G,\mathbb{Z})$ is finite, then it is trivial. This statement fails in two dimensions since for a hypertree $S$, the order of $H_1(S,\mathbb{Z})$ can range from $1$ to $\exp(\Theta(n^2))$, see \cite{kalai1983enumeration}. Thus, while the homology of spanning trees is uninteresting, the homology of $2$-dimensional hypertrees is a very rich subject to study.

Kalai's generalization of Cayley's formula \cite{kalai1983enumeration} states that
\[\sum |H_{1}(S,\mathbb{Z})|^2=n^{{n-2}\choose {2}},\]
where the summation is over all the hypertrees $S$ on the vertex set $[n]$. This formula suggests that the natural probability measure on the set of hypertrees is the one where the probability assigned to a hypertree $S$ is \begin{equation}\label{measuredef}
 \frac{|H_{1}(S,\mathbb{Z})|^2}{n^{{n-2}\choose {2}}}.
\end{equation}
It turns out that this measure is a determinantal probability measure \cite{lyons2003determinantal,hough2006determinantal}. Thus, a random hypertree $T_n$ distributed according to \eqref{measuredef} is called a determinantal hypertree. General random determinantal complexes were investigated by Lyons \cite{lyons2009random}. While uniform random spanning trees are well-studied \cite{ald1,ald2,ald3,grimmett1980random,szekeres2006distribution,lyons2017probability}, a theory of determinantal hypertrees started to emerge only recently. The author determined the local weak limit of determinantal hypertrees~\cite{meszaros2022local}, combining this with estimates on the spectrum of the Laplacian matrix of $T_n$, the author proved that $n^{-2}\log |H_1(T_n,\mathbb{Z})|$ converge in probability to a constant~\cite{meszaros2025homology}. Vander Werf~\cite{werf2022determinantal} and the author~\cite{meszaros2023coboundary} investigated various expansion properties of the union of independent copies of $T_n$. Linial and Peled~\cite{linial2019enumeration} provided estimates on the number of hypertrees. All the above mentioned results extend to dimensions larger than $2$. In the $2$-dimensional case, Kahle and Newman~\cite{kahle2022topology} proved that with high probability the fundamental group $\pi_1(T_n)$ is hyperbolic and has
cohomological dimension 2. Moreover, the author provided lower bounds on $\dim H_1(T_n,\mathbb{F}_2)$~\cite{meszaros20242}, and proved Corollary~\ref{cory1} below in the special case of $p=2$~\cite{meszaros2024bounds}. See also~\cite{meszaros2023cohen,lee2025distribution} for results on random matrix models inspired by determinantal hypertrees. 


\begin{theorem}\label{thm1}
 Let $T_n$ be a $2$-dimensional determinantal hypertree on $n$ vertices, and let $G$ be a finite abelian group. Then 
 \[\lim_{n\to\infty}\frac{\log \mathbb{E}|Z^1(T_n,G)|}{n^2}=0.\]
\end{theorem}

\begin{corr}\label{cory1}
 For any prime $p$,
 \[\frac{\dim H_1(T_n,\mathbb{F}_p)}{n^2}\]
 converges to zero in probability.
\end{corr}

For a group $\Gamma$, let $\mg(\Gamma)$ be the minimum number of generators of $\Gamma$. 
\begin{corr}\label{cory2}
The random variable
 \[\frac{\mg( H_1(T_n,\mathbb{Z}))}{n^2}\]
converges to zero in probability. 
\end{corr}

The same results hold of the \emph{$1$-out $2$-complex} $S_2(n,1)$ which is defined as follows. We start with the complete graph on the vertex set $[n]$, and independently for each edge $\{u,v\}\in {{[n]}\choose{2}}$, we choose a third vertex $w$ from the set $[n]\setminus \{u,v\}$ uniformly at random and add the triangular face $\{u,v,w\}$ to the complex. If a triangular face is chosen at multiple edges, we only keep one copy of that face. 

\begin{theorem}\label{thm2}
Consider the $1$-out $2$-complex $S_2(n,1)$. For any finite abelian group $G$, 
 \[\lim_{n\to\infty}\frac{\log \mathbb{E}|Z^1(S_2(n,1),G)|}{n^2}=0.\]

For any prime $p$,
 \[\frac{\dim H_1(S_2(n,1),\mathbb{F}_p)}{n^2}\]
 converges to zero in probability.

Moreover,
\[\frac{\mg( H_1(S_2(n,1),\mathbb{Z}))}{n^2}\]
converges to zero in probability. 

\end{theorem}

Linial and Peled \cite{linial2019enumeration} proved the second statement of Theorem~\ref{thm2} in the case where the field~$\mathbb{F}_p$ is replaced with~$\mathbb{R}$. 
See \cite{linial2016phase,linial2006homological,rc3,rc4,rc5,rc6,rc7,rc8,rc9,rc10,bobrowski2018topology,bobrowski2022random} for further results on the homology of random simplicial complexes.

\subsection{Cochain graphons and the large deviation principle for random cochains}
In this section, we describe a limit theory for $1$-dimensional cochains with coefficients from a finite abelian group $G$. There are many somewhat similar generalizations of the notion of graphons \cite{janson2008graph,kunszenti2022multigraph,kolossvary2011multigraph,lovasz2010limits,kunszenti2019uniqueness}, but none of these meets our needs perfectly. 

Let $S_G=[0,1]^2\times G$. We endow $S_G$ with the product measure of the Lebesgue measure on $[0,1]^2$ and the counting measure on $G$. Thus, for $1\le q\le \infty$, we can speak about $L^{q}(S_G)$, and we can also define the $L^q$ norm $\|W^G\|_q$ of any measurable map $W^G:S_G\to \mathbb{R}$.

We say that $W^G:S_G\to \mathbb{R}$ is symmetric if 
\[W^G(x,y,g)=W^G(y,x,-g)\text{ for all }(x,y,g)\in S_G.\]

The set of cochain kernels is defined as
\[\mathcal{W}^G=\{W^G\in L^\infty(S_G)\,:\, W^G\text{ is symmetric}\}.\]

The set of cochain graphons is defined as
\[\mathcal{W}_0^G=\{W^G\in \mathcal{W}^G\,:\, 0\le W^G\le 1\}.\]

These definitions can be motivated as follows. Let $K_n$ be the complete graph on $[n]=\{1,2,\dots,n\}$ viewed as a $1$-dimensional simplicial complex. Then a cochain $f\in C^1(K_n,G)$ can be considered as a function $f:\{(u,v)\in [n]^2\,:\,u\neq v\}\to G$ such that $f(u,v)=-f(v,u)$ for all $(u,v)\in [n]^2$, $u\neq v$. One can consider the graph of this $f$, that is,
\[\gr f=\{(u,v,f(u,v))\,:\,(u,v)\in [n]^2,\,u\neq v\}\subset [n]^2\times G.\]

We define $W^G_f:S_G\to [0,1]$ as a continuous version of the indicator function of $\gr f$, namely, let
\[W^G_f(x,y,g)=\begin{cases}
1&\text{if $\lceil nx\rceil\neq \lceil ny\rceil$ and $f(\lceil nx\rceil, \lceil ny\rceil)=g$, and}\\
0&\text{otherwise.}
\end{cases}\]

It is straightforward to check that $W^G_f\in \mathcal{W}_0^G$. Thus, $C^1(K_n,G)$ can be embedded into $\mathcal{W}_0^G$ for all $n$. 

In this paper, our goal is to understand the asymptotic behavior of certain random cochains using cochain graphons. Let us define these random cochains. Let $\nu$ be a symmetric non-degenerate probability distribution on $G$, that is, \[\nu\in \mathbb{R}^G \text{ such that $\sum_{g\in G} \nu(g)=1$, $\nu(g)>0$ and $\nu(g)=\nu(-g)$ for all $g\in G$.}\] Let $F_n=F_{n,\nu}\in C^1(K_n,G)$ be a random cochain such that $(F_n(u,v))_{1\le u<v\le n}$ are independent, and \[\mathbb{P}(F_n(u,v)=g)=\nu(g)\text{ for all $1\le u<v\le n$ and $g\in G$.}\] 
Due to the symmetry of $\nu$, we have the following invariance property of $F_n$.

\begin{lemma}\label{distsym}
Let $\pi$ be a permutation of $[n]$. For any $f\in C^1(K_n,G)$, we define $f^\pi\in C^1(K_n,G)$ by $f^\pi(u,v)=f(\pi(u),\pi(v))$. Then $F_n$ and $F_n^\pi$ have the same distribution.
\end{lemma}

We will prove a large deviation estimate for $F_n$, but in order to do this first we need to define a topology on $\mathcal{W}^G_0$.

For $W^G:S_G\to\mathbb{R}$ and $g\in G$, we define $W^g:[0,1]^2\to\mathbb{R}$ by $W^g(x,y)=W^G(x,y,g)$.

For a bounded measurable function $W:[0,1]^2\to\mathbb{R}$, we define the cut norm of $W$ as
\[\|W\|_{\square}=\sup_{S,T} \left|\int_{S\times T} W(x,y) \,dxdy\right|,\]
where the supremum is over all measurable subsets $S$ and $T$ of $[0,1]$.

The cut norm of a cochain kernel $W^G$ is defined as
\[\|W^G\|_{\square}=\sum_{g\in G}\|W^g\|.\]

Let $S_{[0,1]}$ be the set of invertible measure preserving maps from $[0,1]$ to itself. Given $\varphi\in S_{[0,1]}$ and $W^G\in \mathcal{W}^G$, we define $(W^G)^\varphi\in \mathcal{W}^G$ by
\[(W^G)^{\varphi}(x,y,g)=W^G(\varphi(x),\varphi(y),g).\]

The cut distance of two cochain kernels $V^G,W^G\in \mathcal{W}^G$ is defined as
\[\delta_\square(V^G,W^G)=\inf_{\varphi\in S_{[0,1]}} \|V^G-(W^G)^\varphi\|_\square.\]

The cut distance only gives a pseudometric, since different cochain kernels can have distance zero. Identifying the cochain kernels at distance zero, we obtain the set $\widetilde{\mathcal{W}}^G$ of unlabeled cochain kernels. The equivalence class of a cochain kernel $W^G$ will be denoted by $\widetilde{W}^G$. The set of unlabeled cochain graphons $\widetilde{\mathcal{W}}_0^G$ is defined analogously. By a closed subset of $\widetilde{\mathcal{W}}^G_0$, we mean a closed subset of the metric space $(\widetilde{\mathcal{W}}^G_0,\delta_\square)$. Note that $\widetilde{\mathcal{W}}_0^G$ is compact, see Lemma~\ref{lemmacompact}.

Let \[\mathcal{W}^G_{00}=\left\{W^G\in\mathcal{W}_0^G\,:\,\sum_{g\in G} W^g(x,y)=1\text{ for almost all $(x,y)\in [0,1]^2$}\right\}.\]

For a $W^G\in \mathcal{W}_0^G$, we define the rate function

\begin{equation}\label{ratefunctiondef}I_{\nu}(W^G)=
\begin{cases}
\mathlarger{\frac{1}2\int_{[0,1]^2} \sum_{g\in G}W^g(x,y)\log \frac{W^g(x,y)}{\nu(g)} dxdy}&\text{if }W^G\in \mathcal{W}^G_{00},\\
\infty&\text{otherwise.}
\end{cases}\end{equation}

We prove in Section~\ref{secweakLDP} that the function $I_\nu$ is well-defined on $\widetilde{\mathcal{W}}_0^G$ and is lower semicontinuous with respect to the cut metric.

We made all the necessary preparations in order to be able to state the large deviation principle for the random cochain $F_{n,\nu}$, which is a generalization of the large deviation principle for the Erd\H{o}s-R\'enyi random graph by Chatterjee and Varadhan \cite{chatterjee2011large}.

\begin{theorem}\label{largediv0}
 Let $\nu$ be a symmetric non-degenerate probability distribution on $G$. Consider the random cochain $F_{n,\nu}$, and the corresponding random cochain graphon $R_n^G=W^G_{F_{n,\nu}}$. 
 \begin{enumerate}[(a)]
 \item For any closed subset $\widetilde{F}$ of $\widetilde{\mathcal{W}}_0^G$, we have
 \[\limsup_{n\to\infty} \frac{1}{n^2}\log \mathbb{P}(\widetilde{R}_n^G\in \widetilde{F})\le -\inf_{\widetilde{W}^G\in \widetilde{F}} I_\nu(\widetilde{W}^G).\]

 \item For any open subset $\widetilde{E}$ of $\widetilde{\mathcal{W}}_0^G$, we have
 \[\liminf_{n\to\infty} \frac{1}{n^2}\log \mathbb{P}(\widetilde{R}_n^G\in \widetilde{E})\ge -\inf_{\widetilde{W}^G\in \widetilde{E}} I_\nu(\widetilde{W}^G).\]
 \end{enumerate}
 
\end{theorem}

\subsection{Open questions}

It is conjectured that for every prime $p$, the size of the $p^\infty$-torsion of $H_1(T_n,\mathbb{Z})$ is of constant order, more precisely:
\begin{conjecture}
Let $\Gamma_{n,p}$ be the $p^\infty$-torsion of $H_1(T_n,\mathbb{Z})$. Then $\Gamma_{n,p}$ is tight, that is, given any $\varepsilon>0$, there is a $K$, such that
\[\mathbb{P}(|\Gamma_{n,p}|>K)<\varepsilon\text{ for all large enough }n.\]
This would also imply that $\dim H_1(T_n,\mathbb{F}_p)$ is tight.
\end{conjecture}

Kahle and Newman~\cite{kahle2022topology} had an even stronger conjecture, namely, they conjectured that for any prime $p$, the $p^\infty$-torsion $\Gamma_{n,p}$ of $H_1(T_n,\mathbb{Z})$ converges to the \emph{Cohen-Lenstra distribution}. By now we know that this stronger conjecture is not true for $p=2$ \cite{meszaros20242}, but the case $p>2$ is still open. More precisely, we have the following conjecture. 
\begin{conjecture}[Kahle and Newman~\cite{kahle2022topology}]
Let $p$ be an odd prime, and let $G$ be a finite abelian $p$-group, then
\[\lim_{n\to\infty} \mathbb{P}(\Gamma_{n,p}\cong G)=\frac{1}{|\Aut(G)|}\prod_{j=1}^{\infty}\left(1-p^{-j}\right).\]
\end{conjecture}
This conjecture would imply that
\[\lim_{n\to\infty} \mathbb{P}(\dim H_1(T_n,\mathbb{F}_p)=r)=p^{-r^2} \prod_{j=1}^{r} \left(1-p^{-j}\right)^{-2} \prod_{j=1}^{\infty}\left(1-p^{-j}\right).\]
The same should be also true if we consider the uniform distribution on the hypertrees, see~\cite{kahle2020cohen} for some numerical evidence. See the survey of Wood~\cite{wood2022probability} for more information on the Cohen-Lenstra heuristics. For a toy model motivated by determinantal hypertrees, the Cohen-Lenstra limiting distribution was established by the author for $p\ge 5$ \cite{meszaros2023cohen}, see also \cite{lee2025distribution} for generalizations.

In light of Corollary~\ref{cory2}, the following conjecture is also interesting.
\begin{conjecture}
Let $\pi_1(T_n)$ be the fundamental group of $T_n$, then
\[\frac{\mg( \pi_1(T_n))}{n^2}\]
converges to zero in probability. 
\end{conjecture}

Let us define the $2$-dimensional Linial-Meshulam complex $X_2\left(n,\frac{c}n\right)$ as a random complex on $n$ vertices with a complete $1$-skeleton such that each triangular face is present with probability $\frac{c}n$ independently. Linial and Peled~\cite{linial2016phase} proved that for any fixed $c>0$, the normalized Betti numbers $n^{-2} \dim H_1\left(X_2\left(n,\frac{c}n\right),\mathbb{R}\right)$ converge in probability to a constant. It is conjectured that the same is true in positive characteristic, that is, for any prime $p$, $n^{-2} \dim H_1\left(X_2\left(n,\frac{c}n\right),\mathbb{F}_p\right)$ converge in probability to the same constant. {\L}uczak and Peled~\cite{rc5} had the even stronger conjecture that $H_1\left(X_2\left(n,\frac{c}n\right),\mathbb{Z}\right)$ has trivial torsion with high probability away from the critical density $c_*$. It might be possible to use the methods of the present paper to determine the limit of $n^{-2} \log_p \mathbb{E}|Z_1(X_2\left(n,\frac{c}n\right),\mathbb{F}_p)|$. However, this will not give us the typical behavior of $\dim Z_1(X_2\left(n,\frac{c}n\right),\mathbb{F}_p)$, because due to the fluctuations of $\dim Z_1(X_2\left(n,\frac{c}n\right),\mathbb{F}_p)$, the quantity $n^{-2} \log_p\mathbb{E}|Z_1(X_2\left(n,\frac{c}n\right),\mathbb{F}_p)|$ will be much larger than $n^{-2}\mathbb{E}\dim Z_1(X_2\left(n,\frac{c}n\right),\mathbb{F}_p)$. See for example~\cite[Section 2.1]{ayre2020satisfiability}, where the same phenomenon is discussed in a slightly different setting. Thus, to understand the mod $p$ Betti numbers of the $2$-dimensional Linial-Meshulam complexes, it seems like we need additional ideas.

It would be also interesting to develop a higher dimensional analogue of our method, and generalize our results to $d$-dimensional determinantal hypertrees and $1$-out $d$-complexes.

\medskip 

\textbf{Acknowledgement:} 
The author was supported by the NKKP-STARTING 150955 project and the Marie Sk\l{}odowska-Curie Postdoctoral Fellowship "RaCoCoLe".








\section{Preliminaries on cochain graphons}

\subsection{Convolution of cochain graphons}


Given $V,W\in L^{\infty}([0,1]^2)$, their operator product $V\circ W$, which is a bounded measurable function from $[0,1]^2$ to $\mathbb{R}$, is defined as
\[(V\circ W)(x,y)=\int_0^1 V(x,t)W(t,y)dt.\]

 Given two cochain graphons $V^G,W^G$ their convolution $V^G*W^G$ is defined as the function $U^G\in L^{\infty}(S_G)$ satisfying
 \[U^g=\sum_{h\in G} V^h\circ W^{g-h}.\]

 The following lemma motivates the definition of $V^G*W^G$. The proof of this lemma is straightforward from the definitions.

 \begin{lemma}\label{cochainconv}
Let $f\in C^1(K_n,G)$. For $u,w\in [n]$ and $g\in G$, let
\[P(u,w,g)=\left|\{v\in [n]\setminus\{u,w\}\,:\, f(u,v)+f(v,w)=g\}\right|.\]
Then for all $(x,y)\in (0,1]^2$ and $g\in G$, we have
\[(W^G_f*W^G_f)(x,y,g)=\frac{P(\lceil nx\rceil, \lceil ny \rceil, g)}{n}.\]
\end{lemma}

For a function $W:[0,1]^2\to \mathbb{R}$, let $W^T:[0,1]^2\to \mathbb{R}$ be defined as $W^T(x,y)=W(y,x)$.

\begin{lemma}\label{ConvinW00} Assume that $W^G\in \mathcal{W}^G_{00}$, then $W^G*W^G\in \mathcal{W}^G_{00}$.
\end{lemma}
\begin{proof}
It is clear that $W^G*W^G$ is non-negative. 

We have
\[\sum_{g\in G} (W^G*W^G)^g=\sum_{g\in G}\sum_{h\in G} W^h\circ W^{g-h}=\left(\sum_{g\in G}W^g\right)\circ \left(\sum_{g\in G}W^g\right)=\mathbbm{1}\circ \mathbbm{1}=\mathbbm{1},\]
where is $\mathbbm{1}$ is the constant $1$ function on $[0,1]^2$.

 To prove the symmetry of $W^G*W^G$, observe that
\[((W^G*W^G)^g)^T=\sum_{h\in G} (W^h\circ W^{g-h})^T=\sum_{h\in G} (W^{g-h})^T\circ (W^{h})^T=\sum_{h\in G} W^{h-g}\circ W^{-h}=(W^G*W^G)^{-g}.\]
\end{proof}
\subsection{A few continuity results}

In this section, we will rely on a few simple lemmas proved in \cite{meszaros2024bounds}. 


(Note that in \cite{meszaros2024bounds}, elements of $L^\infty([0,1]^2)$ were often assumed to be symmetric, but this assumption is not really used. Thus, all the results cited in this section are true without assuming the symmetry of the functions $W:[0,1]^2\to\mathbb{R}$. Also, for certain results in \cite{meszaros2024bounds} the boundedness assumption can be replaced by integrability.)


\begin{lemma}\label{limitofproduct}
(\cite[Lemma 2.2]{meszaros2024bounds})Let $V,V_1,V_2,\dots$ and $W,W_1,W_2\dots$ be elements of $L^{\infty}([0,1]^2)$. Assume that there is a constant $C$ such that $\|V_n\|_\infty\le C$ for all $n$. If $\lim_{n\to\infty} V_n=V$ and $\lim_{n\to\infty} W_n=W$ in the cut norm, then $\lim_{n\to\infty} V_n\circ W_n=V\circ W$ in the $L^1$ norm. 
\end{lemma}

\begin{lemma}\label{limitofconvolution}
Let $V^G,V_1^G,V_2^G,\dots$ and $W^G,W_1^G,W_2^G\dots$ be elements of $\mathcal{W}^G_0$. Assume that there is a constant $C$ such that $\|V_n^G\|_\infty\le C$ for all $n$. If $\lim_{n\to\infty} V_n^G=V^G$ and $\lim_{n\to\infty} W_n^G=W^G$ in the cut norm, then $\lim_{n\to\infty} V_n^G* W_n^G=V^G* W^G$ in the $L^1$ norm. 
\end{lemma}
\begin{proof}
By Lemma~\ref{limitofproduct}, we have $\lim_{n\to\infty} V_n^g\circ W_n^{g-h}=V^g\circ W^{g-h}$ in the $L^1$ norm for all $g,h\in G$. Thus, the statement follows.
\end{proof}

Given two measurable functions $V,W:[0,1]^2\to\mathbb{R}$, we define
\[\langle V,W\rangle=\int_{[0,1]^2} V(x,y)W(x,y) dxdy,\]
whenever the integral above exists.


\begin{lemma}\label{lemmaangle}
 (\cite[Lemma 2.4]{meszaros2024bounds}) Let $V,V_1,V_2,\dots$ be elements of $L^\infty([0,1]^2)$ and $W,W_1,W_2\dots$ be elements of $L^1([0,1]^2)$. Assume that there is a constant~$C$ such that $\|V\|_\infty\le C$ and $\|V_n\|_\infty\le C$ for all $n$. If $\lim_{n\to\infty} V_n=V$ in the cut norm and $\lim_{n\to\infty} W_n=W$ in the $L^1$ norm, then $\lim_{n\to\infty} \langle V_n,W_n\rangle =\langle V,W\rangle$.
\end{lemma}

Given two measurable functions $V^G,W^G: S_G\to\mathbb{R}$, we define
\[\langle V^G,W^G\rangle=\sum_{g\in G} \langle V^g,W^g\rangle,\]
whenever the right hand side exists.

As a direct consequence of Lemma~\ref{lemmaangle}, we obtain the following statement.

\begin{lemma}\label{lemmaangleWG}
 Let $V^G,V^G_1,V^G_2,\dots$ be elements of $L^\infty(S_G)$ and $W^G,W^G_1,W^G_2\dots$ be elements of $L^1(S_G)$. Assume that there is a constant~$C$ such that $\|V^G\|_\infty\le C$ and $\|V_n^G\|_\infty\le C$ for all $n$. If $\lim_{n\to\infty} V_n^G=V^G$ in the cut norm and $\lim_{n\to\infty} W_n^G=W^G$ in the $L^1$ norm, then $\lim_{n\to\infty} \langle V_n^G,W_n^G\rangle =\langle V^G,W^G\rangle$.
\end{lemma}

Given a measurable function $f:\mathbb{R}\to\mathbb{R}$ and a cochain kernel $W^G\in \mathcal{W}^G$, $f\circ W^G$ is a symmetric measurable function from $S_G$ to $\mathbb{R}$. If $f$ is bounded, then $f\circ W^G\in \mathcal{W}^G$.

The proof of the next lemma is straightforward.

\begin{lemma}\label{Lipschitz}
Let $W_n^G\in L^{1}(S_G)$. Assume that $\lim_{n\to\infty} W_n^G=W^G$ in $L^1$ norm. If $f:\mathbb{R}\to\mathbb{R}$ is a bounded Lipschitz function, then $\lim_{n\to\infty} f\circ W_n^G=f\circ W^G$ in $L^1$ norm. 
\end{lemma}

\section{The proof of Theorem~\ref{thm1} and its corollaries}
\subsection{Bounding the number of cocycles}

For a simplicial complex $K$, let $K(i)$ be the set of $i$-dimensional faces of~$K$.

Let $Y\subset {{[n]}\choose {3}}$. For any $\tau\in {{[n]}\choose {2}}$, we define
\[t_Y(\tau)=|\{\sigma\in Y\,:\, \tau\subset \sigma\}|.\]

The next lemma was proved in \cite[Lemma 3.2]{meszaros2024bounds}.
\begin{lemma}\label{upperb}
Let $Y\subset {{[n]}\choose {3}}$, then
\[\log \mathbb{P}(T_n(2)\subset Y)\le (n-2)\log (n)+\left(1-\frac{2}n\right)\sum_{\tau\in {{[n]}\choose{2}}} \log\left(\frac{t_Y(\tau)}{n}\right).\]
\end{lemma}

 For $f\in C^1(K_n,G)$, let us consider $Y=Y_f$ defined as
\begin{align}
\label{eqYdef}Y&=\left\{\{u,v,w\}\in {{[n]}\choose{3}}\,:\,f(u,v)+f(v,w)+f(w,u)=0\right\}\\
&=\left\{\{u,v,w\}\in {{[n]}\choose{3}}\,:\,f(u,v)+f(v,w)=f(u,w)\right\}.\nonumber
\end{align}

Clearly,
\begin{equation}\label{cohgraph}
 \mathbb{P}\left(f\in Z^1(T_n,G)\right)=\mathbb{P}\left(T_n(2)\subset Y \right).
\end{equation}

Combining this with Lemma~\ref{upperb}, we see that
\begin{equation}\label{logPbound}\log \mathbb{P}\left(f\in Z^1(T_n,G)\right)\le (n-2)\log (n)+\left(1-\frac{2}n\right)\sum_{\tau\in {{[n]}\choose{2}}} \log\left(\frac{t_Y(\tau)}{n}\right).
\end{equation}

With the notations of Lemma~\ref{cochainconv},
\[t_Y(\{u,w\})=P(u,w,f(u,w))\quad\text{ for all }\{u,w\}\in {{[n]}\choose{2}}.\]

 Using Lemma~\ref{cochainconv}, we see that
\begin{align*}\log\left(\frac{t_Y(\{u,w\})}n\right)&=n^2\int_{\frac{u-1}n}^{\frac{u}n}\int_{\frac{w-1}n}^{\frac{w}n} \log\left((W^G_f*W^G_f)(x,y,f(u,w))\right) dx dy\\&=n^2\sum_{g\in G} \int_{\frac{u-1}n}^{\frac{u}n}\int_{\frac{w-1}n}^{\frac{w}n} W^G_f(x,y,g)\log\left((W^G_f*W^G_f)(x,y,g)\right) dxdy.
\end{align*}
Here we define $0\cdot \log(0)$ to be $0$ as it is standard practice in information theory. 

Let \[E=\bigcup_{v\in [n]} \left(\frac{v-1}n,\frac{v}n\right]^2.\]

Note that $W^G_f$ vanishes on $E\times G$, so
\begin{equation}\label{vanishonE}\sum_{g\in G} \int_E W^G_f(x,y,g)\log\left((W^G_f*W^G_f)(x,y,g)\right) dxdy=0.\end{equation}

Thus,
\begin{align}\label{eq12}
\sum_{\tau\in{{[n]}\choose{2}}} \log\left(\frac{t_Y(\tau)}{n}\right)&=\frac{1}2\sum_{\substack{u,w\in [n]\\u\neq w}}\log\left(\frac{t_Y(\{u,w\})}{n}\right)\\&=\frac{n^2}2 \sum_{\substack{u,w\in [n]\\u\neq w}}\sum_{g\in G} \int_{\frac{u-1}n}^{\frac{u}n}\int_{\frac{w-1}n}^{\frac{w}n} W^G_f(x,y,g)\log\left((W^G_f*W^G_f)(x,y,g)\right) dxdy\nonumber\\&=\frac{n^2}2 \sum_{g\in G} \int_{[0,1]^2\setminus E} W^G_f(x,y,g)\log\left((W^G_f*W^G_f)(x,y,g)\right) dxdy\nonumber\\&=\frac{n^2}2 \sum_{g\in G} \int_{[0,1]^2} W^G_f(x,y,g)\log\left((W^G_f*W^G_f)(x,y,g)\right) dxdy\nonumber\\&=\frac{n^2}2\left\langle W^G_f,\log\circ (W^G_f*W^G_f)\right\rangle.\nonumber
\end{align}
where in the second to last step, we used \eqref{vanishonE}.

Combining this with \eqref{logPbound}, we obtain that 
\begin{equation}\label{upperbf}\log \mathbb{P}\left(f\in Z^1(T_n,\mathbb{F}_2)\right)\le (n-2)\log (n)+\frac{n^2}2\left(1-\frac{2}n\right)b(W^G_f),\end{equation}
where 
\begin{equation}\label{fdef}b(W^G)=\left\langle W^G,\log\circ (W^G*W^G)\right\rangle.
\end{equation}

\subsection{Applying the large deviation principle}

Let $\overline{\nu}$ be the uniform measure on $G$. For a cochain graphon $W^G\in\mathcal{W}_0^G$, we define
\begin{equation*}H(W^G)=\log |G|-2I_{\overline{\nu}}(W^G),\end{equation*}
where $I_{\overline{\nu}}(W^G)$ was defined in \eqref{ratefunctiondef}. Since $I_{\overline{\nu}}$ is well-defined on $\widetilde{\mathcal{W}}_0^G$, so is $H$.

Combining the definitions of $I_{\overline{\nu}}$ and $H$, we see that
\[H(W^G)=
\begin{cases}
\mathlarger{\int_{[0,1]^2} \sum_{g\in G}-W^g(x,y)\log W^g(x,y) dxdy}&\text{if }W^G\in \mathcal{W}^G_{00},\\
-\infty&\text{otherwise.}
\end{cases}\]

For a subset $\widetilde{F}$ of $\widetilde{\mathcal{W}}_0^G$, let
\[\mathcal{C}(n,\widetilde{F})=\{f\in C^1(K_n,G)\,:\,\widetilde{W}_f^G\in \widetilde{F}\}.\]

Letting $\widetilde{R}_n^G=\widetilde{W}^G_{F_{n,\overline{\nu}}}$, we see that
\[|\mathcal{C}(n,\widetilde{F})|=|G|^{{n}\choose{2}}\mathbb{P}(\widetilde{R}^G_n\in \widetilde{F})=|G|^{(1+o(1))\frac{n^2}2}\mathbb{P}(\widetilde{R}^G_n\in \widetilde{F}).\]

Thus, as a straightforward corollary of Theorem~\ref{largediv0}, we obtain the following statement.
\begin{lemma}\label{logcount}
Let $\widetilde{F}$ be a closed subset of $\widetilde{\mathcal{W}}_0^G$. Then
\[\limsup_{n\to\infty} \frac{1}{n^2}\log |\mathcal{C}(n,\widetilde{F})|\le \frac{1}2\sup_{\widetilde{V}^G\in \widetilde{F}} H(\widetilde{V}^G).\]

\end{lemma}

Recall that we defined $b$ in \eqref{fdef}.

\begin{lemma}\label{semic0}
The function $b$ is upper semicontinuous on $\mathcal{W}_0^G$ with respect to the cut norm.
\end{lemma}
\begin{proof}
For $k>0$, we define $\ell_k:[0,1]\to \mathbb{R}$ as $\ell_k(x)=\max(-k,\log(x))$. Given a cochain graphon $W^G\in \mathcal{W}^G_0$, let
\[b_k(W^G)=\langle W^G, \ell_k\circ (W^G*W^G)\rangle.\]

By the monotone convergence theorem, we have $b(W^G)=\inf_{k>0} b_k(W^G)$. So it is enough to prove that $b_k$ is continuous for all $k>0$. Let $W^G_1,W^G_2,\dots$ be a sequence of cochain graphons converging to the cochain graphon $W^G$ in the cut norm. Using Lemma~\ref{limitofconvolution}, we see that $\lim_{n\to\infty} W_n^G*W_n^G=W^G*W^G$ in the $L^1$ norm. Since $\ell_k$ is a Lipschitz function, it follows from Lemma~\ref{Lipschitz} that $\lim_{n\to\infty}\ell_k\circ (W_n^G*W_n^G)=\ell_k\circ (W^G*W^G)$ in the $L^1$ norm. Finally, we can apply Lemma~\ref{lemmaangleWG} to conclude that $\lim_{n\to\infty}\langle W^G_n, \ell_k\circ (W_n^G*W_n^G)\rangle=\langle W^G, \ell_k\circ (W^G*W^G)\rangle$. Therefore, $b_k$ is indeed continuous.
\end{proof}

\begin{remark} The function $b$ is not continuous on $\mathcal{W}_0^G$ with respect to the cut norm as the following example shows. Let $G=\mathbb{Z}/2\mathbb{Z}$, and let us define $W^G_n$ as follows, $W_n^1$ be the indicator function of $[0,n^{-1}]^2$ and let $W_n^0=\mathbbm{1}-W_n^1$, then $W_n^G$ converge to $W^G$, where $W^0=\mathbbm{1}$ and $W^1=0$. We have $b(W_n^G)=-\infty$ for all $n$, but $b(W^G)=0$. 
\end{remark}
\begin{lemma}\label{semic}
The function $b$ is a well-defined upper semicontinuous function on $\widetilde{\mathcal{W}}^G_0$.
\end{lemma}
\begin{proof}
To prove that $b$ is well-defined we need to show that if $V^G$ and $W^G$ are two cochain graphons such that $V^G\in \widetilde{W}^G$, then $b(V^G)=b(W^G)$. Since $V^G\in \widetilde{W}^G$, we can choose a sequence $\varphi_n\in S_{[0,1]}$ such that $(W^G)^{\varphi_n}$ converge to $V^G$ in the cut norm. It is clear from the definition that $b((W^G)^{\varphi_n})=b(W^G)$. Combining this with Lemma~\ref{semic0}, it follows that $b(V^G)\ge b(W^G)$. Since $W^G\in \widetilde{V}^G$ also holds, the same argument gives that $b(W^G)\ge b(V^G)$. Thus, $b(V^G)=b(W^G)$. So $b$ is indeed well-defined. The semicontinuity is clear from Lemma~\ref{semic0}. 
\end{proof}

For $u\ge 0$, let
\[\widetilde{F}_u=\left\{\widetilde{W}^G\in \widetilde{\mathcal{W}}^G_0\,:\,b(\widetilde{W}^G)\ge -u\right\}.\]

\begin{lemma}\label{LDFu}
For $u\ge 0$, we have
\[\limsup_{n\to\infty} \frac{1}{n^2}\log |\mathcal{C}(n,\widetilde{F}_u)|\le \frac{u}2.\]

\end{lemma}
\begin{proof}
 Using Lemma~\ref{semic}, we see that $\widetilde{F}_u$ is closed. Thus, by Lemma~\ref{logcount}, we need to show that $H(W^G)\le u$ for all $W^G\in \mathcal{W}^G_0$ such that $b(W^G)\ge -u$. Clearly, $H(W^G)\le u$ if $W^G\notin \mathcal{W}^G_{00}$. Thus, the statement will follow once we prove that $b(W^G)+H(W^G)\le 0$ for all $W^G\in \mathcal{W}^G_{00}$. This inequality holds, because
 \begin{align*}
 b(W^G)+H(W^G)&=\int_{[0,1]^2}\sum_{g\in G} W^g(x,y)\log\left(\frac{(W^G*W^G)(x,y,g)}{W^g(x,y)}\right) dx dy\\&=\int_{[0,1]^2}-\dkl\left((W^g(x,y))_{g\in G}\, \Big|\Big|\, ((W^G*W^G)(x,y,g))_{g\in G}\right)dxdy\\&\le 0,
 \end{align*}
 where $\dkl$ denotes the Kullback–Leibler divergence and the last inequality follows from Gibbs' inequality. Here we need to use Lemma~\ref{ConvinW00} to see that $((W^G*W^G)(x,y,g))_{g\in G}$ is indeed a probability measure on $G$.
\end{proof}


\subsection{The proof of Theorem~\ref{thm1}}

Let $k$ be a large positive integer, and let $\varepsilon=\frac{\log |G|}k$. For $i=0,1,2,\dots,k-1$, let
\[\widetilde{L}_i=\left\{\widetilde{W}^G\in \widetilde{\mathcal{W}}_0^G\,:\: -i\varepsilon \ge b(\widetilde{W}^G)>-(i+1)\varepsilon \right\},\]
and let $\widetilde{L}_{k}=\left\{\widetilde{W}^G\in \widetilde{\mathcal{W}}_0^G\,:\: -k\varepsilon \ge b(\widetilde{W}^G) \right\}$. 

Since $b$ is non-positive, for every $n$, $(\mathcal{C}(n,\widetilde{L}_i))_{i=0}^{k}$ gives us a partition of $C^1(K_n,G)$. Using~\eqref{upperbf}, we see that if $n$ is large enough, then for all $f\in \mathcal{C}(n,\widetilde{L}_i)$, we have
\begin{align}\label{Pbound}\mathbb{P}\left(f\in Z^1(T_n,G)\right)&\le (n-2)\log (n)+\frac{n^2}2\left(1-\frac{2}n\right)b(W_f^G)\\&\le(n-2)\log (n)-\frac{n^2}2\left(1-\frac{2}n\right)i\varepsilon\nonumber\\&\le -\frac{n^2(i-1)\varepsilon}2. \nonumber
\end{align}

Let $0\le i\le k-1$. We have $\widetilde{L}_i\subset \widetilde{F}_{(i+1)\varepsilon}$, so Lemma~\ref{LDFu} gives us that for all large enough $n$, we have
\begin{equation}\label{GLi}\log |\mathcal{C}(n,\widetilde{L}_i)|\le \log |\mathcal{C}(n,\widetilde{F}_{(i+1)\varepsilon})|\le \frac{n^2(i+2)\varepsilon}2.
\end{equation}

Since $|C^1(K_n,G)|\le |G|^{\frac{n^2}2}$, \eqref{GLi} also holds for $i=k$. 

Thus, combining \eqref{Pbound} and \eqref{GLi}, we see that for all large enough $n$, we have
\begin{align*}\mathbb{E} |Z^1(T_n,G)|&\le\sum_{i=0}^k \sum_{f\in \mathcal{C}(n,\widetilde{L}_i)}\mathbb{P}\left(f\in Z^1(T_n,G)\right)\\&\le \sum_{i=0}^k \exp\left(\frac{n^2(i+2)\varepsilon}2\right)\exp\left(-\frac{n^2(i-1)\varepsilon}2\right)\\&\le (k+1)\exp\left(2\varepsilon n^2 \right).
\end{align*}

Tending to infinity with $k$, the statement follows.

\subsection{The proof of Corollary~\ref{cory1}}

Note that $H_1(T_n,\mathbb{F}_p)\cong H^1(T_n,\mathbb{F}_p)$, see \cite[Chapter 3.1]{hatcher2000algebraic}. It will be more convenient to work with $H^1(T_n,\mathbb{F}_p)$.

Using Markov's inequality, we see that for any $\varepsilon>0$, we have
\begin{align*}\mathbb{P}(\dim H^1(T_n,\mathbb{F}_p)\ge 2\varepsilon n^2)&=\mathbb{P}(| H^1(T_n,\mathbb{F}_p)|\ge p^{2\varepsilon n^2})\\&\le \mathbb{P}(| Z^1(T_n,\mathbb{F}_p)|\ge p^{2\varepsilon n^2}) \\&\le \frac{\mathbb{E} |Z^1(T_n,\mathbb{F}_p)|}{p^{2\varepsilon n^2}}. \end{align*}

By Theorem~\ref{thm1}, we have $\mathbb{E} |Z^1(T_n,\mathbb{F}_p)|\le p^{\varepsilon n^2}$ for all large enough $n$. Thus, for all large enough $n$, we have
\[\mathbb{P}(\dim H^1(T_n,\mathbb{F}_p)\ge 2\varepsilon n^2)\le p^{-\varepsilon n^2},\]
which tends to $0$. Tending to $0$ with $\varepsilon$, the statement follows.

\subsection{The proof of Corollary~\ref{cory2}}

\begin{lemma}\cite[Lemma 2.1]{rc3}\label{torsionbound}
If $X$ is a simplicial complex on $n$ vertices, then
\[|\torsion H_1(X,\mathbb{Z})|\le 3^{\frac{n^2}4}.\]
\end{lemma}
\begin{lemma}\label{mgbound}
For any $2$-dimensional simplicial complex $X$ and $k\ge 2$, we have
\[\mg (H_1(X,\mathbb{Z}))\le \max_{p\le k\text{ is a prime}} \dim H_1(X,\mathbb{F}_p)+\frac{\log |\torsion H_1(X,\mathbb{Z})|}{\log(k)}.\]
\end{lemma}
\begin{proof}
Let $\mg_p$ be the minimum number of generators of the $p$-torsion of $H_1(X,\mathbb{Z})$. Then $p^{\mg_p}\le |\torsion H_1(X,\mathbb{Z})|$. Since $H_0(X,\mathbb{Z})$ is free, it follows from the universal coefficient theorem that
\[\dim H_1(X,\mathbb{F}_p)=\dim H^1(X,\mathbb{F}_p)=\dim H_1(X,\mathbb{Q})+\mg_p.\]
Thus, for any $p>k$, we have
\[\dim H_1(X,\mathbb{F}_p)\le \dim H_1(X,\mathbb{F}_2)+\log_k |\torsion H_1(X,\mathbb{Z})|.\]
The statement follows by combining this estimate with the fact that
\[\mg (H_1(X,\mathbb{Z}))=\sup_{p\text{ is a prime}} \dim H_1(X,\mathbb{F}_p).\qedhere\]
\end{proof}

Combining Lemma~\ref{torsionbound} and Lemma~\ref{mgbound}, we see that
\[\mg(H_1(X,\mathbb{Z}))\le \max_{p\le k\text{ is a prime}} \dim H_1(X,\mathbb{F}_p)+\frac{\log(3)n^2}{4\log(k)}.\]

Combining this with Corollary~\ref{cory1}, we see that 
\[\lim_{n\to\infty}\mathbb{P}\left(\mg(H_1(T_n,\mathbb{Z}))\ge \frac{\log(3)n^2}{\log(k)}\right)=0.\]

Tending to infinity with $k$, the statement follows.

\subsection{The proof of Theorem~\ref{thm2}}

The proof of Theorem~\ref{thm2} is almost identical to the proof of Theorem~\ref{thm1}, so we only sketch it.

It is easy to see and it was also observed by Linial and Peled \cite[Section 5]{linial2019enumeration} that 
\[\mathbb{P}\left(f\in Z^1(S_2(n,1),G)\right)=\prod_{\tau\in{{[n]}\choose{2}}} \frac{t_Y(\tau)}{n-2},\]
where $Y$ is defined as in \eqref{eqYdef}. 

Thus,
\begin{align*}\log \mathbb{P}\left(f\in Z^1(S_2(n,1),G)\right)&={{n}\choose{2}}\log\left(1+\frac{2}{n-2}\right)+\sum_{\tau\in{{[n]}\choose{2}}} \log \frac{t_Y(\tau)}{n}\\&\le \frac{n(n-1)}{n-2}+\frac{n^2}2b(W_f^G),\end{align*}
where the last inequality follows from \eqref{eq12}.

The proof can be finished along the lines of the proofs of Theorem~\ref{thm1}, Corollary~\ref{cory1} and Corollary~\ref{cory2}.



\section{The proof of the large deviation principle}

In this section, we prove Theorem~\ref{largediv0}. We mostly follow the argument of Chatterjee and Varadhan~\cite{chatterjee2011large}. However, we decided to provide more details compared to \cite{chatterjee2011large}.

\subsection{Regularity lemma}\label{regularitysection}

Let $\mathcal{P}$ be a partition of $[0,1]$ into finitely many measurable sets with positive measure. Given an integrable function $W:[0,1]^2\to\mathbb{R}$, the stepping $S_{\mathcal{P}}W:[0,1]^2\to \mathbb{R}$ of $W$ is defined by setting
\[(S_{\mathcal{P}}W)(x,y)=\frac{1}{|S||T|}\int_{S\times T} W(s,t) dsdt,\]
for all $x\in S\in \mathcal{P}$ and $y\in T\in \mathcal{P}$.

We say that $W:[0,1]^2\to\mathbb{R}$ is $\mathcal{P}$-measurable if $W$ is constant on all the sets of the form $S\times T$, where $S,T\in \mathcal{P}$. A cochain graphon $W^G$ is $\mathcal{P}$-measurable if $W^g$ is $\mathcal{P}$-measurable for all $g\in G$.

The proof of the next lemma is an easy exercise.

\begin{lemma}\label{stepnorm}
Assume that $W:[0,1]^2\to\mathbb{R}$ is $\mathcal{P}$-measurable, then
\[\|W\|_{\square}=\max_{\mathcal{S},\mathcal{T}\subset\mathcal{P}}\left|\int_{(\cup \mathcal{S})\times (\cup\mathcal{T})} W(s,t) ds dt\right|\]

\end{lemma}

\begin{lemma}\cite[Lemma 9.12.]{lovasz2012large}\label{factortwo}
Let $W_1,W_2:[0,1]^2\to\mathbb{R}$ be integrable functions, assume that $W_2$ is $\mathcal{P}$-measurable. Then
\[\|W_1-S_{\mathcal{P}}W_1\|_\square \le 2\|W_1-W_2\|_\square.\]

\end{lemma}

For $W^G\in L^1(S_G)$, we define $S_{\mathcal{P}} W^G\in L^1(S_G)$ by setting
\[S_{\mathcal{P}} W^G(x,y,g)=S_{\mathcal{P}} W^g(x,y) \text{ for all }g\in G.\]
It is easy to check that if $W^G\in \mathcal{W}_0^G$, then $S_{\mathcal{P}} W^G\in \mathcal{W}_0^G$.

For $A\subset\{1,2,\dots,n\}$, we define
\[A_{/n}=\cup_{a\in A} \left(\frac{a-1}n,\frac{a}n\right].\]
For a partition $\mathcal{P}$ of $\{1,2,\dots,n\}$, we define the partition $\mathcal{P}_{/n}$ of $[0,1]$ as
\[\mathcal{P}_{/n}=\{A_{/n}\,:\,A\in \mathcal{P} \}.\]

Let $\mathcal{F}_n=\mathcal{P}_{/n}$, where $\mathcal{P}$ is the partition of $\{1,2,\dots,n\}$ into singletons. 

Given an $n\times n$ matrix $M$ over $\mathbb{R}$, we define the $\mathcal{F}_n$-measurable function $W_M:[0,1]^2\to \mathbb{R}$ by
\begin{equation}\label{MWM}
W_M(x,y)=M(\lceil nx\rceil,\lceil ny\rceil).
\end{equation}
The map $M\mapsto W_M$ gives a bijection between the sets of $n\times n$ matrices and $\mathcal{F}_n$-measurable functions from $[0,1]^2$ to $\mathbb{R}$.

Let us define
\[\|M\|_{\square}=\frac{1}{n^2} \max_{S,T\subset [n]}\left|\sum_{i\in S}\sum_{j\in T} M(i,j)\right|.\]

Then, using Lemma~\ref{stepnorm}, we see that
\begin{equation}\label{MWMnorm}
 \|M\|_{\square}=\|W_M\|_{\square}.
\end{equation}

Although Chatterjee and Varadhan \cite{chatterjee2011large} used the Szemer\'edi's original version \cite{szemeredi1978regular} of the regularity lemma, for our purposes the following variant will be more convenient. 

\begin{lemma}[Frieze and Kannan \cite{frieze1999quick}]\label{friezekannan}
For every $\varepsilon>0$, there is a $k$ with the following property. Given any $n\times n$ matrix $M$, there is a partition $\mathcal{P}$ of $[n]$ with at most $k$ parts and a matrix $M_0$ such that $M_0$ is constant on the set $S\times T$ for all $S,T\in \mathcal{P}$ and
\[\|M-M_0\|_{\square}\le \varepsilon \max_{1\le i, j\le n} |M(i,j)|.\]
\end{lemma}
\begin{lemma}\label{stepregularity}
For any $\varepsilon>0$, there is a $k$ with the following property. Let $W^G$ be an $\mathcal{F}_n$-measurable cochain graphon, then there is a partition $\mathcal{P}$ of $\{1,2,\dots,n\}$ with at most $k$ parts such that 
\[\|W^G-S_{\mathcal{P}_{/n}}W^G\|_\square\le \varepsilon.\]
\end{lemma}
\begin{proof}
Combining Lemma~\ref{friezekannan} with the bijection given in \eqref{MWM} and \eqref{MWMnorm}, we see that there is a $k_0$ such that for all $g\in G$, we have a partition $\mathcal{P}^g$ of $[n]$ and a $\mathcal{P}^g_{/n}$-measurable function $W_0^g:[0,1]^2\to\mathbb{R}$ such that

\[\|W^g-W_0^g\|_\square\le \frac{\varepsilon}{2|G|}.\]
Let $\mathcal{P}$ be the smallest common refinement of $\mathcal{P}^g$ $(g\in G)$. Clearly, $W_0^g$ is also $\mathcal{P}_{/n}$-measurable, so by Lemma~\ref{factortwo}, we see that
\[\|W^g-S_{\mathcal{P}_{/n}}W^g\|_\square\le \frac{\varepsilon}{|G|}\text{ for all $g$}.\]
Therefore,
\[\|W^G-S_{\mathcal{P}_{/n}}W^G\|_\square\le \varepsilon.\]
Since the partition $\mathcal{P}$ has at most $k_0^{|G|}$ parts, the statement follows. 
\end{proof}

Using Lebesgue differentiation theorem, one can prove the following lemma.
\begin{lemma}\label{Lebesgue}
Let $W^G\in L^1(S_G)$. Then $S_{\mathcal{F}_n} W^G$ converge to $W^G$ both almost everywhere and in~$L^1$.
\end{lemma}

\begin{lemma}\label{graphonreg}
For any $\varepsilon>0$, there is a $k$ with the following property. Let $W^G$ be a cochain graphon, then there is a partition $\mathcal{P}$ of $[0,1]$ with at most $k$ parts such that 
\[\|W^G-S_{\mathcal{P}}W^G\|_\square\le \varepsilon.\]
\end{lemma}
\begin{proof}
By Lemma~\ref{Lebesgue}, if we choose $n$ large enough, then
\[\|W^G-S_{\mathcal{F}_n} W^G\|_\square \le \|W^G-S_{\mathcal{F}_n} W^G\|_1\le \frac{\varepsilon}2.\]

Then Lemma~\ref{stepregularity} provides us a $k$ such that there is a partition $\mathcal{Q}$ of $[n]$ with at most $k$ parts such that $\|S_{\mathcal{F}_n} W^G-S_{\mathcal{Q}_{/n}}S_{\mathcal{F}_n} W^G\|_\square\le\frac{ \varepsilon}{2}.$ Note that $n$ might depend on $W^G$, but $k$ is independent from $W^G$. Observing that $S_{\mathcal{Q}_{/n}}S_{\mathcal{F}_n} W^G=S_{\mathcal{Q}_{/n}} W^G$, we have
\[\|W^G-S_{\mathcal{Q}_{/n}} W^G\|_\square\le \|W^G-S_{\mathcal{F}_n} W^G\|_\square+\|S_{\mathcal{F}_n} W^G-S_{\mathcal{Q}_{/n}} W^G\|_\square\le \varepsilon.\qedhere\]
\end{proof}

\begin{lemma}\label{lemmacompact}
The space $\widetilde{\mathcal{W}}_0^G$ is compact with the cut metric. 
\end{lemma}
\begin{proof}
This can be proved using the regularity lemma (Lemma~\ref{graphonreg}) along the lines of the proof of the compactness of the space of graphons \cite[Theorem 9.23.]{lovasz2012large}.
\end{proof}

\subsection{The G\"artner-Ellis Theorem}

Let $\mathcal{X}$ be a Hausdorff real topological vector space, and let $\mathcal{X}^*$ be its continuous dual. Let $X_n$ be a sequence of $\mathcal{X}$-valued random variables, and let $\alpha(n)$ be a sequence of positive reals tending to infinity.

For $\lambda\in \mathcal{X}^*$, let \[\bar{\Lambda}(\lambda)=\limsup_{n\to\infty}\frac{1}{\alpha(n)} \log \mathbb{E}\exp\left(\alpha(n)\lambda(X_n)\right).\]

For $x\in \mathcal{X}$, let
\[\bar{\Lambda}^*(x)=\sup_{\lambda\in \mathcal{X}^*} \left(\lambda(x)-\bar{\Lambda}(\lambda)\right).\]

The next theorem can be found in \cite[Theorem 4.5.3]{dembolarge}.
\begin{theorem}\label{GEthm}
Let $F$ be a compact subset of $\mathcal{X}$, then
\[\limsup_{n\to\infty} \,\frac{1}{\alpha(n)}\log \mathbb{P}(X_n\in F)\le -\inf_{x\in F} \bar{\Lambda}^*(x).\]
\end{theorem}

The next theorem can be found in \cite[Theorem 4.5.20 (b)]{dembolarge}. Note that we replaced the assumptions of \cite[Theorem 4.5.20]{dembolarge} with some stronger assumptions for the sake of simplicity. 

\begin{theorem}\label{GEthmb}
Assume that there is a compact subset $K$ of $\mathcal{X}$ such that $X_n\in K$ for all $n$, and
\[\lim_{n\to\infty}\frac{1}{\alpha(n)} \log \mathbb{E}\exp\left(\alpha(n)\lambda(X_n)\right)\]
exists and finite for all $\lambda\in \mathcal{X}^*$. Let $\mathcal{F}$ be the set of exposed points of $\bar{\Lambda}^*$, and let $E$ be an open subset of $\mathcal{X}$. Then
\[\liminf_{n\to\infty} \,\frac{1}{\alpha(n)}\log \mathbb{P}(X_n\in E)\ge -\inf_{x\in E\cap \mathcal{F}} \bar{\Lambda}^*(x).\]
\end{theorem}

\subsection{Weak large deviation principle}\label{secweakLDP}

Each $\phi\in L^1 (S_G)$ determines a linear functional $Z_{\phi}:\mathcal{W}^G\to \mathbb{R}$ by the formula
\[Z_{\phi}(W^G)=\sum_{g\in G}\int_{[0,1]^2} W(x,y,g)\phi(x,y,g) dxdy.\]

Let \begin{align*}L^1_{\sym} (S_G)&=\{W^G\in L^1(S_G)\,:\,W^G\text{ is symmetric}\},\text{ and }\\
 \mathcal{H}&=\{Z_\phi\,:\,\phi\in L^1_{\sym} (S_G)\}.
 \end{align*}

Let us consider $\mathcal{W}^G$ with the weak topology determined by $\mathcal{H}$, that is, with the coarsest topology such that all the maps $Z_{\phi}\in \mathcal{H}$ are continuous.

Note that it the symmetry of $\phi$ is not a significant restriction as the following straightforward lemma shows.
\begin{lemma}\label{symmetrization}
Let $\phi\in L^1(S_G)$, and let us define
\[\overline{\phi}(x,y,g)=\frac{\phi(x,y,g)+\phi(y,x,-g)}2.\]
Then $\overline{\phi}\in L^1_{\sym} (S_G)$ and $Z_{\phi}=Z_{\overline{\phi}}$. 
\end{lemma}

The next lemma is an easy consequence of Lemma~\ref{lemmaangleWG}. 
\begin{lemma}\label{Lemmacoarser}
 The weak topology on $\mathcal{W}^G_0$ is coarser than the topology given by the cut norm, that is, for all $\phi\in L^1_{\sym}(S_G)$, $Z_\phi$ is a continuous function on $\mathcal{W}^G_0$ with respect to the cut norm. 
\end{lemma}

For $W^G\in \mathcal{W}_0^G$ and $r>0$, let
\[
B_\square(W^G,r)=\{U^G\in \mathcal{W}_0^G \,:\, \|W^G-U^G\|_\square\le r\}.
\]

In the next lemma we collect some of the basic properties of the weak topology on $\mathcal{W}^G_0$.
\begin{lemma}\label{topology}\hfill
\begin{enumerate}[(a)]
\item $\mathcal{W}_0^G$ is weakly compact.
\item For any $W^G\in \mathcal{W}_0^G$ and $r>0$, $B_\square(W^G,r)$ is weakly compact.
\item $(\mathcal{W}^G)^*=\mathcal{H}$, where $(\mathcal{W}^G)^*$ denotes the continuous dual of $\mathcal{W}^G$ with respect to the weak topology.
\end{enumerate}
\end{lemma}
\begin{proof}
(a) Let $\Delta=\{(x,y)\,:\,0\le x\le y\le 1\}$, and $S^{1/2}_G=\Delta\times G$. Let us endow $S^{1/2}_G$ with two times the product of the Lebesgue measure on $\Delta$ and the counting measure on $G$. Then $L^1_{\sym}(S_G)\cong L^1(S_G^{1/2})$ and $\mathcal{W}^G\cong L^\infty(S_G^{1/2})$. Since $L^1(S_G^{1/2})^*=L^\infty(S_G^{1/2})$, it follows that $\mathcal{W}^G$ with the weak topology is isomorphic to $L^\infty(S_G^{1/2})$ with the weak$^*$ topology. Thus, $\mathcal{W}_0^G$ is weakly compact by the Banach–Alaoglu theorem.

(b) Let $\mathcal{Y}$ be the set of all function $\phi\in L_1(S_G)$ which can be obtained as follows. For all $g\in G$, let $S_g$ and $T_g$ be a measurable subset of $[0,1]$ and let $s_g\in\{+1,-1\}$, and let
\[\phi(x,y,g)=s_g \mathbbm{1}((x,y)\in S_g\times T_g).\]
Using the notations of Lemma~\ref{symmetrization}, for all $U^G\in \mathcal{W}^G$, we have
\[\|U^G\|_\square=\sup_{\phi\in \mathcal{Y}} Z_{\phi}(U^G)=\sup_{\phi\in \mathcal{Y}} Z_{\overline{\phi}}(U^G).\]
Thus,
\[B_\square(W^G,r)=\mathcal{W}_0^G\cap\left(W^G+ \bigcap_{\phi\in \mathcal{Y}} Z_{\overline{\phi}}^{-1}((-\infty,r])\right).\]
Therefore, $B_\square(W^G,r)$ is closed since it is the intersection of closed sets. Combining this with part (a), $B_\square(W^G,r)$ is a closed subset of a compact set, thus, it is compact.

(c) See \cite[Theorem B.8]{dembolarge}.
\end{proof}

Let $R_n$ be defined as in Theorem~\ref{largediv0}.

\begin{lemma}\label{explimit}
Let $\phi\in L^1_{\sym} (S_G)$, then
\[\lim_{n\to\infty} \frac{1}{n^2}\log\mathbb{E} \exp(n^2 Z_\phi(R_n))=\frac{1}2\int_{[0,1]^2} \log \left(\sum_{g\in G}\nu(g)\exp\left(2\phi(x,y,g)\right)\right)dxdy.\]

Here the right hand side is finite.
\end{lemma}
\begin{proof}
For $1\le i,j\le n$, let \[D_n(i,j)=\left(\frac{i-1}n,\frac{i}n\right]\times \left(\frac{j-1}n,\frac{j}n\right],\]
and let $E_n=\cup_{i=1}^n \left(\frac{i-1}n,\frac{i}n\right]^2$.

Note that $\left(R_n\restriction \left(D_n(i,j)\cup D_n(j,i)\right)\times G\right)_{1\le i<j\le n}$ are independent. Moreover, $R_n$ vanishes on $E_n\times G$. 

Thus,
\begin{align*}\mathbb{E} \exp(n^2 Z_\phi(R_n))&=\prod_{1\le i<j\le n} \mathbb{E}\exp\left(n^2 \sum_{g\in G} \int_{D_n(i,j)\cup D_n(j,i)} \phi(s,t,g)R_n^g(s,t) dsdt\right)\\&=\prod_{1\le i<j\le n}\sum_{g\in G}\nu(g)\exp\left(n^2 \int_{D_n(i,j)} \left(\phi(s,t,g)+\phi(t,s,-g)\right) dsdt\right). 
\end{align*}
Observe that for any $(x,y)\in D_n(i,j)$, we have \[n^2 \int_{D_n(i,j)} \left(\phi(s,t,g)+\phi(t,s,-g)\right) dsdt=2\phi_n(x,y,g)=2\phi_n(y,x,-g),\] where $\phi_n=S_{\mathcal{F}_n} \phi$. For the definition of $S_{\mathcal{F}_n} \phi$, see Section~\ref{regularitysection}.

Thus,
\begin{multline*}
\log\left(\sum_{g\in G}\nu(g)\exp\left(n^2 \int_{D_n(i,j)} \left(\phi(s,t,g)+\phi(t,s,-g)\right) dsdt\right)\right)\\=\frac{n^2}2 \int_{D_n(i,j)} \log \left(\sum_{g\in G}\nu(g)\exp\left(2\phi_n(x,y,g)\right)\right)+\log \left(\sum_{g\in G}\nu(g)\exp\left(2\phi_n(y,x,-g)\right)\right)dxdy.
\end{multline*}

Therefore,
\[\frac{1}{n^2}\log \mathbb{E} \exp(n^2 Z_\phi(R_n))=\frac{1}2\int_{[0,1]^2\setminus E_n} \log \left(\sum_{g\in G}\nu(g)\exp\left(2\phi_n(x,y,g)\right)\right)dxdy.\]

The map 
\[(x_g)_{g\in G}\mapsto \log \left(\sum_{g\in G}\nu(g)\exp\left(2x_g\right)\right)\]
is Lipschitz, $\phi_n$ converges to $\phi$ in $L^1$ by Lemma~\ref{Lebesgue}, so
\begin{multline*}\lim_{n\to\infty}\int_{[0,1]^2} \log \left(\sum_{g\in G}\nu(g)\exp\left(2\phi_n(x,y,g)\right)\right)dxdy\\=\int_{[0,1]^2} \log \left(\sum_{g\in G}\nu(g)\exp\left(2\phi(x,y,g)\right)\right)dxdy.
\end{multline*}
Also, since the measure of $E_n$ goes to zero it is easy to see that
\[\lim_{n\to\infty}\int_{E_n} \log \left(\sum_{g\in G}\nu(g)\exp\left(2\phi_n(x,y,g)\right)\right)dxdy=0.\]
Thus, the first statement follows.

The second statement follows easily from the fact that $\|\phi\|_1<\infty$ and
\[\log \left(\sum_{g\in G}\nu(g)\exp\left(2\phi(x,y,g)\right)\right)\le 2\max_{g\in G} |\phi(x,y,g)|.\qedhere\]
\end{proof}

Let us consider
\[\bar{\Lambda}^*_{\nu}(W^G)=\sup_{\phi\in L^1_{\sym} (S_G)} \left(Z_\phi(W^G)-\frac{1}2\int_{[0,1]^2} \log \left(\sum_{g\in G}\nu(g)\exp\left(2\phi(x,y,g)\right)\right)dxdy\right).\]

Recall that $I_{\nu}$ was defined in \eqref{ratefunctiondef}.

\begin{lemma}\label{lambdanu}
For all $W^G\in \mathcal{W}^G_0$, we have
\[\bar{\Lambda}^*_{\nu}(W^G)=I_{\nu}(W^G).\]
\end{lemma}
\begin{proof}

First, assume that $W^G\in \mathcal{W}^G_{00}$. 

The next lemma is an easy exercise in calculus.
\begin{lemma}
Let $(w_g)_{g\in G}$ be positive reals such that $\sum_{g\in G}w_g=1$, then
\[\sup_{(x_g)\in \mathbb{R}^{G}} \left(\sum_{g\in G}w_gx_g-\frac{1}2\log \left(\sum_{g\in G}\nu(g)\exp\left(2x_g\right)\right)\right)=\frac{1}2 \sum_{g\in G}w_g\log \frac{w_g}{\nu(g)},\]
and the maximum attained on the set of vectors $\{(x_{C,g})_{g\in G}\,:\,C\in \mathbb{R}\}$, where 
\[x_{C,g}=\frac{1}2\log\left(\frac{w_g}{\nu(g)}\right)+C.\]
\end{lemma}

In fact, the lemma above remains true if we only assume that $w_g$ is nonnegative provided the we interpret $\log(0)$ as $-\infty$, $0\cdot \log(0)$ as $0$, and $\exp(-\infty)$ as $0$.

Therefore it follows easily that
\[\bar{\Lambda}^*_{\nu}(W^G)\le \frac{1}2\int_{[0,1]^2} \sum_{g\in G}W^g(x,y)\log \frac{W^g(x,y)}{\nu(g)} dxdy=I_\nu(W^G).\]

One can define $\phi:S_G\to \mathbb{R}$ by
\begin{equation}\label{phidef}\phi(x,y,g)=\frac{1}2\log \left(\frac{W^g(x,y)}{\nu(g)}\right).
\end{equation}
Since $W^G$ is symmetric and $\nu(g)=\nu(-g)$ for all $g\in G$, we see that $\phi$ is also symmetric. Thus, if $\phi\in L_1(S_g)$, then it follows easily that $\bar{\Lambda}^*_{\nu}= I_\nu(W^G).$

If $\phi\notin L_1(S_g)$, then $\bar{\Lambda}^*_{\nu}= I_\nu(W^G)$ is still true, because we can take a sequence of truncations of $\phi$.

If $W^G\notin \mathcal{W}^G_{00}$, then $I_{\nu}(W^G)=\infty$. For example, let us assume that the set
\[B=\left\{(x,y)\in [0,1]^2\,:\, \sum_{g\in G} W^g(x,y)>1\right\}\]
has positive measure. Let $\phi_C(x,y,g)=\mathbbm{1}((x,y)\in B)\cdot C$. Then
\begin{align*}
\lim_{C\to\infty} &\left(Z_{\phi_C}(W^G)-\frac{1}2\int_{[0,1]^2} \log \left(\sum_{g\in G}\nu(g)\exp\left(2\phi_C(x,y,g)\right)\right)dxdy\right)\\
&=\lim_{C\to\infty} C\int_B \left( \sum_{g\in G} W^g(x,y)-1\right)\, dxdy=\infty,
\end{align*}
which shows that $\bar{\Lambda}^*_{\nu}(W^G)=\infty$. 

The other cases are left to the reader.
\end{proof}

Let
\[\mathcal{W}^G_{00<}=\{W^G\in\mathcal{W}^G_{00}\,:\inf W^G>0\}.\]

\begin{lemma}
Assume that $W^G\in\mathcal{W}^G_{00<}$. Then $W^G$ is an exposed point of $\bar{\Lambda}^*_{\nu}$.
\end{lemma}
\begin{proof}
Let us define $\phi$ as in \eqref{phidef}. Since $W^G\in \mathcal{W}^G_{00<}$, we have $\phi\in L^1_{\sym}(S_G)$. Let $W^G\neq U^G\in \mathcal{W}_{00}^G$, then using Lemma~\ref{lambdanu}, we see that
\begin{align*}
\left(Z_{\phi}(W^G)-\bar{\Lambda}^*_{\nu}(W^G)\right)&-\left(Z_{\phi}(U^G)-\bar{\Lambda}^*_{\nu}(U^G)\right)\\
&=\frac{1}2 \int_{[0,1]^2} \sum_{g\in G} U^g(x,y) \log\left(\frac{U^g(x,y)}{W^g(x,y)}\right)dxdy\\&=\frac{1}2\int_{[0,1]^2}\dkl\left((U^g(x,y))_{g\in G}\, \Big|\Big|\, (W^g(x,y))_{g\in G}\right)dxdy\\&> 0,
 \end{align*}
 where $\dkl$ denotes the Kullback–Leibler divergence and the last inequality follows from Gibbs' inequality. If $U^G\notin \mathcal{W}_{00}^G$, then 
 \[\left(Z_{\phi}(W^G)-\bar{\Lambda}^*_{\nu}(W^G)\right)-\left(Z_{\phi}(U^G)-\bar{\Lambda}^*_{\nu}(U^G)\right)>0\]
 obviously holds, since $\bar{\Lambda}^*_{\nu}(U^G)=+\infty$.
\end{proof}

\begin{lemma}\label{lemmainfrestricted}
Let $E$ be a weakly open subset of $\mathcal{W}_0^G$, then
\[\inf_{W^G\in E} \bar{\Lambda}^*_{\nu}(W^G)=\inf_{W^G\in E\cap \mathcal{W}_{00<}} \bar{\Lambda}^*_{\nu}(W^G).\]
\end{lemma}
\begin{proof}
It is enough to prove that for all $W^G\in E$, we have $\bar{\Lambda}^*_{\nu}(W^G)\ge \inf_{W^G\in E\cap \mathcal{W}_{00<}} \bar{\Lambda}^*_{\nu}(W^G)$. This is clear if $W^G\notin \mathcal{W}_{00}^G$, so we may assume that $W^G\in \mathcal{W}_{00}^G$. For $t\in [0,1]$, let us define $W_t^G\in \mathcal{W}^G_{00}$ by setting
\[W_t^g(x,y)=tW^g(x,y)+(1-t)\frac{1}{|G|}.\]
Note that for $0\le t<1$, we have $W_t^G\in \mathcal{W}^G_{00<}$. Also, $\lim_{t\to 1} W^G_t=W^G$ in $L^{\infty}$ so also in the weak topology. Thus, since $E$ is open, we have a $0\le t_0<1$ such that $W^t\in E\cap \mathcal{W}_{00<}^G$ for all $t_0<t<1$. It is also straightforward to prove that
\[\lim_{t\to 1}\bar{\Lambda}^*_{\nu}(W_t^G)=\bar{\Lambda}^*_{\nu}(W^G),\]
so the statement follows.
\end{proof}

Now we are able to prove the following weak large deviation principle.
\begin{lemma}\label{LweakLDP}
Let $\nu$ be a symmetric non-degenerate probability distribution on $G$. Consider the random cochain $F_{n,\nu}$, and the corresponding random cochain graphon $R_n^G=W^G_{F_{n,\nu}}$. 
\begin{enumerate}[(a)]
\item For any weakly compact subset $F$ of $\mathcal{W}_0^G$, we have
 \[\limsup_{n\to\infty} \frac{1}{n^2}\log \mathbb{P}(R_n^G\in F)\le -\inf_{{W}^G\in {F}} I_\nu({W}^G).\]
\item For any weakly open subset $E$ of $\mathcal{W}_0^G$, we have
 \[\liminf_{n\to\infty} \frac{1}{n^2}\log \mathbb{P}(R_n^G\in E)\ge -\inf_{{W}^G\in {E}} I_\nu({W}^G).\]
 \end{enumerate}
\end{lemma}
\begin{proof}
The upper bound in part (a) follows by combining Lemma~\ref{lambdanu} with the G\"artner-Ellis Theorem (Theorem~\ref{GEthm}).

The lower bound in part (b) follows by combining Lemma~\ref{explimit}, Lemma~\ref{lambdanu}, Lemma~\ref{lemmainfrestricted} with the G\"artner-Ellis Theorem (Theorem~\ref{GEthmb}).
\end{proof}

In the next two lemmas we discuss the continuity properties of the function $I_{\nu}$.

\begin{lemma}
$I_{\nu}$ is a lower semicontinuous function on $\mathcal{W}_0$ with respect to the cut norm. 
\end{lemma}
\begin{proof}

For $\phi\in L^1_{\sym} (S_G)$, the function \[Z_\phi(W^G)-\frac{1}2\int_{[0,1]^2} \log \left(\sum_{g\in G}\nu(g)\exp\left(2\phi(x,y,g)\right)\right)dxdy\] is a continuous function on $\mathcal{W}_0$ with respect to the cut norm by Lemma~\ref{Lemmacoarser}. The statement follows from Lemma~\ref{lambdanu} and the fact that the supremum of continuous functions is lower semicontinuous.
\end{proof}

\begin{lemma}\label{lowersemi}
The function $I_{\nu}$ is a well-defined lower semicontinuous function on $\widetilde{\mathcal{W}}^G_0$.
\end{lemma}
\begin{proof}
This statement can be proved the same way as Lemma~\ref{semic0}.
\end{proof}

\subsection{Strong large deviation principle -- The proof of Theorem~\ref{largediv0}}
Since $\widetilde{W}^G_{0}$ is compact by Lemma~\ref{lemmacompact}, we see that $\widetilde{F}$ is not only closed, but it is also compact.

For any $\eta>0$, let us define
\[B_\square(\widetilde{F},\eta)=\{\widetilde{W}^G\in \widetilde{\mathcal{W}}^G_0\,:\,\delta_\square(\widetilde{W}^G,\widetilde{U}^G)\le \eta \text{ for some }\widetilde{U}^G\in \widetilde{F}\}.\]

Let $\varepsilon>0$. Recall that $\widetilde{F}$ is compact and $I_{\nu}$ is lower semicontinuous by Lemma~\ref{lowersemi}. Thus, there is an $\eta>0$ such that
\[\inf_{\widetilde{W}^G\in B_\square(\widetilde{F},6\eta)} I_{\nu}(\widetilde{W}^G) \ge \inf_{\widetilde{W}^G\in \widetilde{F}} I_{\nu}(\widetilde{W}^G)-\varepsilon.\]

By Lemma~\ref{stepregularity} there is a $k$ such that, for all $n$, there is a partition $\mathcal{P}$ of $[n]$ with at most $k$ parts such that
\[\|S_{\mathcal{P}_{/n}} R_n^G-R_n^G\|_\square\le \eta.\]
Of course, the choice $\mathcal{P}$ might depend on $R_n$. There are at most $k^n$ partitions of $[n]$ with at most $k$ parts. Thus, there is a deterministic partition $\mathcal{Q}^{(n)}$ of $[n]$ with $\ell_n$ parts such that $\ell_n\le k$ and 
\[\mathbb{P}\left(\widetilde{R}_n^G\in \widetilde{F}\text{ and }\|S_{\mathcal{Q}^{(n)}_{/n}} R_n^G-R_n^G\|_\square\le {\eta}\right)\ge \frac{\mathbb{P}(\widetilde{R}_n^G\in \widetilde{F})}{k^n}.\]
By passing to a subsequence, we may assume that $\ell_n=\ell$ for some $\ell$. Let $s_1^{(n)},s_2^{(n)},\dots,s_\ell^{(n)}$ be the sizes of the parts of $\mathcal{Q}^{(n)}$, and let $t^{(n)}_i=\sum_{j=1}^i s_j^{(n)}$, and let us define the following partition of $[n]$:

\[\mathcal{T}^{(n)}=\{\{t_{i-1}^{(n)}+1,t_{i-1}^{(n)}+2,\dots,t_i^{(n)}\}\,:\,1\le i\le \ell\}.\]

By the distributional symmetry of $R_n$ given in Lemma~\ref{distsym}, we have
\begin{align*}\mathbb{P}&\left(\widetilde{R}_n^G\in \widetilde{F}\text{ and }\|S_{\mathcal{T}^{(n)}_{/n}} R_n^G-R_n^G\|_\square\le {\eta}\right)\\&\qquad=\mathbb{P}\left(\widetilde{R}_n^G\in \widetilde{F}\text{ and }\|S_{\mathcal{Q}^{(n)}_{/n}} R_n^G-R_n^G\|_\square\le {\eta}\right)\\&\qquad\ge \frac{\mathbb{P}(\widetilde{R}_n^G\in \widetilde{F})}{k^n}.
\end{align*}

By passing to a subsequence, we may assume that for all $0\le i\le \ell$, we have 
\[\lim_{n\to\infty} \frac{t_i^{(n)}}n=t_i.\]

Given a function $M:[\ell]^2\times G\to [0,1]$, let us define $W^G_{M,n},W^G_M:S_G\to [0,1]$ by setting
\[W^G_{M,n}(x,y,g)=M(i,j,g)\text{ for all }(x,y)\in \left[\frac{t_{i-1}^{(n)}}n,\frac{t_{i}^{(n)}}n\right)\times \left[\frac{t_{j-1}^{(n)}}n,\frac{t_{j}^{(n)}}n\right),\]

and

\[W^G_{M}(x,y,g)=M(i,j,g)\text{ for all }(x,y)\in \left[t_{i-1},t_i\right)\times \left[t_{j-1},t_j\right).\]

Let $\mathcal{M}$ be the set of all functions from $[\ell]^2\times G$ to $\{\eta,2\eta,\dots,\lfloor \eta^{-1}\rfloor \eta\}$ satisfying the symmetry constraint that
\[M(i,j,g)=M(j,i,-g)\text{ for all }(i,j,g)\in [\ell]^2\times G.\]

The set
\[\cup_{M\in \mathcal{M}} B_\square(W^G_{M,n},\eta)\]
contains all $\mathcal{T}^{(n)}_{/n}$ measurable cochain graphons. Thus, by the pigeonhole principle, there is an $M_n\in \mathcal{M}$ such that

\[\mathbb{P}\left(\widetilde{R}_n^G\in \widetilde{F},\quad \|S_{\mathcal{T}^{(n)}_{/n}} R_n^G-R_n^G\|_\square\le {\eta}\text{ and }S_{\mathcal{T}^{(n)}_{/n}} R_n^G\in B_\square(W^G_{M_n,n},\eta)\right)\ge \frac{\mathbb{P}(\widetilde{R}_n^G\in \widetilde{F})}{k^n|\mathcal{M}|}.\]

By passing to a subsequence, we may assume that $M_n=M$ for all $n$. It is easy to see that for all large enough $n$, we have
\[\|W^G_{M,n}-W^G_M\|_\square\le \|W^G_{M,n}-W^G_M\|_1\le \eta.\]

Therefore, assuming that $n$ is large enough, on the event $\|S_{\mathcal{T}^{(n)}_{/n}} R_n^G-R_n^G\|_\square\le {\eta}$ and $S_{\mathcal{T}^{(n)}_{/n}} R_n^G\in B_\square (W^G_{M,n},\eta)$, 
we have
\begin{align*}
 \|R_n^G-W^G_{M}\|_\square\le \|R_n^G-S_{\mathcal{T}^{(n)}_{/n}} R_n^G\|_\square+ \|S_{\mathcal{T}^{(n)}_{/n}} R_n^G-W^G_{M,n}\|_\square+\|W^G_{M,n}-W^G_M\|_\square\le 3\eta.
\end{align*}
Thus, for all large enough $n$, we have
\begin{align*}\mathbb{P}&\left(\widetilde{R}_n^G\in \widetilde{F}\text{ and }\|R_n^G-W^G_{M}\|_\square\le 3\eta\right)\\&\ge \mathbb{P}\left(\widetilde{R}_n^G\in \widetilde{F},\quad \|S_{\mathcal{T}^{(n)}_{/n}} R_n^G-R_n^G\|_\square\le {\eta}\text{ and }S_{\mathcal{T}^{(n)}_{/n}} R_n^G\in B_\square(W^G_{M_n,n},\eta)\right)\\&\ge \frac{\mathbb{P}(\widetilde{R}_n^G\in \widetilde{F})}{k^n|\mathcal{M}|}.\end{align*}

Clearly, $\lim_{n\to\infty} \frac{1}{n^2} \log(k^n|\mathcal{M}|)=0$, thus
\begin{align}\limsup_{n\to\infty}\frac{1}{n^2}\log\mathbb{P}\left(R_n^G\in B_\square(W^G_M,3\eta) \right)&\ge \limsup_{n\to\infty}\frac{1}{n^2}\log\mathbb{P}\left(\widetilde{R}_n^G\in \widetilde{F}\text{ and }\|R_n^G-W^G_{M}\|_\square\le 3\eta\right)\nonumber\\&\ge \limsup_{n\to\infty}\frac{1}{n^2}\log\mathbb{P}\left(\widetilde{R}_n^G\in \widetilde{F}\right). \label{eq11} \end{align}

Clearly, we may assume that $\mathbb{P}\left(\widetilde{R}_n^G\in \widetilde{F}\text{ and }\|R_n^G-W^G_{M}\|_\square\le 3\eta\right)>0$ for some $n$. In that case $\widetilde{W}_M^G\in B_\square(\widetilde{F},3\eta)$. Thus, for any $W^G\in B_\square(W^G_M,3\eta)$, we have $\widetilde{W}^G\in B_\square(\widetilde{F},6\eta)$, and so $I_{\nu}(W^G)\ge \inf_{\widetilde{W}^G\in \widetilde{F}} I_{\nu}(\widetilde{W}^G)-\varepsilon$. Note that $B_\square(W^G_M,3\eta)$ is weakly compact by Lemma~\ref{topology}. Thus, we can combine our earlier observation with the weak large deviation principle given in Lemma~\ref{LweakLDP},

\[\limsup_{n\to\infty}\frac{1}{n^2}\log\mathbb{P}\left(R_n^G\in B_\square(W^G_M,3\eta) \right)\le -\inf_{W^G\in B_\square(W^G_M,3\eta)}I_{\nu}(W^G)\le -\inf_{\widetilde{W}^G\in \widetilde{F}} I_{\nu}(\widetilde{W}^G)+\varepsilon.\]

Combining this with \eqref{eq11}, part (a) of Theorem~\ref{largediv0} follows. Part (b) can be proved along the lines of the proof of the lower bound in \cite{chatterjee2011large} with minimal technical changes, so we omit the details. 

\bibliography{references}
\bibliographystyle{plain}

\bigskip

\noindent Andr\'as M\'esz\'aros, \\
{\tt meszaros@renyi.hu}\\
HUN-REN Alfr\'ed R\'enyi Institute of Mathematics,\\
Budapest, Hungary

\end{document}